\documentclass[11pt]{article}      
\usepackage{geometry}                       
\usepackage{graphicx}               
\usepackage{amsfonts}
\usepackage{mathrsfs}
\usepackage{amsmath}
\usepackage{amsthm}                             
\usepackage{amssymb}
\usepackage{cite}

\def\ds{\displaystyle}
\def\={\buildrel \triangle \over =}
\def\a{\alpha}

\def\d{\delta}

\def\l{\lambda}

\def\n{\nabla}
\def\si{\sigma}

\def\f{\varphi}

\def\o{\omega}

\def\ns{\noalign{\ss} }

\def\G{\Gamma}
\def\D{\Delta}

\def\O{\Omega}

\def\ms{\medskip}

\def\q{\quad}
\def\qq{\qquad}

\def\dbR{{\mathbb{R}}}

\def\3n{\negthinspace \negthinspace \negthinspace }
\def\2n{\negthinspace \negthinspace }
\def\1n{\negthinspace }

\def\cA{{\cal A}}

\def\cF{{\cal F}}

\def\pa{\partial}

\def\cd{\cdot}

\def\Re{{\mathop{\rm Re}\,}}

\def\|{||}
\def\({\Big (}
\def\){\Big )}
\def\[{\Big[}
\def\]{\Big]}
\def\be{\begin{equation}}
\def\bel{\begin{equation}\label}
\def\ee{\end{equation}}
\def\bt{\begin{theorem}}
\def\bcd{\begin{condition}}
\def\ecd{\end{condition}}
\def\et{\end{theorem}}
\def\bc{\begin{corollary}}
\def\ec{\end{corollary}}
\def\bde{\begin{definition}}
\def\ede{\end{definition}}
\def\bl{\begin{lemma}}
\def\el{\end{lemma}}
\def\bp{\begin{proposition}}
\def\ep{\end{proposition}}
\def\br{\begin{remark}}
\def\er{\end{remark}}
\def\ba{\begin{array}}
\def\ea{\end{array}}
\def\ed{\end{document}}
\def\ns{\noalign{\ms}}
\def\ds{\displaystyle}

\def\Om{\Omega}

\newtheorem{lemma}{Lemma}[section]
\newtheorem{remark}{Remark}[section]

\newtheorem{theorem}{Theorem}[section]
\newtheorem{corollary}{Corollary}[section]

\newtheorem{definition}{Definition}[section]
\newtheorem{proposition}{Proposition}[section]
\newtheorem{condition}{Condition}[section]

\allowdisplaybreaks

\title{\bf Logarithmic Stability for Coefficients Inverse
Problem of Coupled Schr\"{o}dinger Equations}
\author{Fangfang Dou\thanks{School of Mathematical Sciences, University of Electronic Science and Technology of China, Chengdu, China. Email: fangfdou@uestc.edu.cn.} \ and Masahiro Yamamoto\thanks{Department of Mathematical Sciences, The University of Tokyo, 3-8-1 Komaba, Meguro, Tokyo 153, Japan. Email: myama@ms.u-tokyo.ac.jp.}\ \thanks{Peoples’Friendship University of Russia (RUDN University), 6 Miklukho-Maklaya St, Moscow, 117198, Russian Federation}}
\date{}                            

\begin{document}

\maketitle
\begin{abstract}
In this paper, we study an inverse coefficients
problem for two coupled Schr\"{o}dinger
equations with an observation of one component
of the solution. The observation is done in a
nonempty open subset of the domain where the
equations hold. A logarithmic type stability
result is obtained. The main method is based on
the Carleman estimate for coupled
Schr\"{o}dinger equations and coupled heat
equations, and  the Fourier-Bros-Iagolnitzer transform.
\end{abstract}

\noindent {\bf Keywords:} logarithmic stability, 
coefficients inverse problem, coupled Schr\"{o}dinger equations, 
Carleman estimate

\section{Introduction}

Let $T>0$ and $\O\subset\dbR^3$ be a nonempty
bounded domain with smooth boundary and let 
$i = \sqrt{-1}$. Consider the following coupled
Schr\"{o}dinger equations:
\begin{equation}\label{eq1}
\left\{
\begin{array}{ll}
\ds i\partial_t y_1+\Delta
y_1+a_{11}(x)y_1+a_{12}(x)y_2=0 &\text{ in
} \Omega\times(0,T),\\
\ns\ds i\partial_t y_2+\Delta
y_2+a_{21}(x)y_1+a_{22}(x)y_2=0 &\text{ in
} \Omega\times(0,T),\\
\ns\ds y_1=0, y_2=0 &  \text{ on
}\Gamma\times(0,T),\\
\ns\ds y_1(x,0)=y_{10}, \; y_2(x,0)=y_{20}&
\text{ in }\Omega.
\end{array}
\right.
\end{equation}
System \eqref{eq1} is a useful model for describing 
molecular multiphoton transitions
induced by a laser (e.g.\cite{BA1993,GM1992}),
where  $a_{11}(x)$ and $a_{22}(x)$ are
field-free molecular electronic potentials, and
$a_{12}(x)$ and $a_{21}(x)$ are
radiation-molecule interactions. In physical
models, usually, the radiation-molecule
interactions can be deduced a priori while the
field-free molecular electronic potentials
should be determined a posteriori.

Let $\o$ be a nonempty open subset of $\O$. In
this paper, we study the following inverse
problems:

\vspace{0.1cm}

{\bf Problem (IP)} Can one recover the
field-free molecular electronic potentials
$(a_{11}, a_{22})$ from suitable observation of
$y_1$ on $[0,T]\times \o$?

\vspace{0.1cm}

Here the word ``recover" means two issues:
One is that the
observation determines the potentials uniquely.
The other is to find an algorithm to compute the potentials
efficiently. 

A stability estimate
\begin{equation}\label{7.30-eq1}
\|(a_{11},a_{22})\| \leq C\|y_1|_{\o}\|
\end{equation}
with suitable norms under suitable boundedness conditions is not only 
important theoretically but also essential for the second issue: it can 
guarantee the convergence of the numerical
algorithm for computing $(a_{11},a_{22})$.

Inequalities in the type of \eqref{7.30-eq1} for
Schr\"odinger equations were studied extensively
(e.g.
\cite{BC2009,BM2008,BKS2016,BP2007,CLG2010,C2012,D2015,KPS,MOR2008,Yuan2010}).
Roughly speaking, the existing works fall
into two categories: one is Lipschitz type
stability when the observation domain fulfills
some geometrically condition (e.g.
\cite{BM2008,BP2007,CLG2010,C2012,D2015,KPS,MOR2008,Yuan2010}),
while the other is logarithmic type stability when the
observation domain is a general nonempty open
subset of the domain or its boundary (e.g.
\cite{BC2009,BKS2016}). For the latter case,
some a priori knowledge about the potential on a
suitable subdomain should be known (see \cite{BC2009}).

A main method for establishing the Lipschitz type
stability is based on
Carleman estimate. On the other hand, the key
method for proving the logarithmic type stability
is a combination of the Carleman estimate and
the Fourier-Bros-Iagolnitzer (F.B.I.)
transformation. For  readers who are not
familiar with the F.B.I. transform, we refer them
to \cite{Delort1992} for an introduction and to
\cite{Phung2001} for the application of F.B.I.
transform to establish observability estimate
for Schr\"odinger equations.

To the best of our knowledge, although there are several interesting works concerning inverse problem for a parabolic system with two components by measurements of one component, for \cite{BCGY2009} as an example, there is no work
on the inverse coefficients problem for
the coupled Schr\"{o}dinger equations with an
observation on one component of the solution. 
Due to the essential difference between these two equations, 
we have to argue independently of \cite{BCGY2009} in the case of parabolic 
systems.
In this paper, we will study this problem by 
the Carleman estimate for Schr\"{o}dinger
equation, coupled heat equations and F.B.I.
transform. Although we borrow some idea in
\cite{BC2009} to prove our main result, since we
study the inverse problem for couple
Schr\"odinger equations with a single
observation on one component of the solution, we
cannot simply mimic the method in \cite{BC2009}
to obtain the desired logarithmic type
stability. Some technical obstacles should be
overcome, as is seen in the proof.

The rest of this paper is organized as follows.
Section 2 is devoted to presenting the main
result while section 3 is devoted to the proof of the main result.


\section{Statement of the main result}


Let $\o_0$ be an open subset of $\O$ such that
there exists a function $\phi\in
C^4(\overline{\Omega})$ satisfying
\begin{equation}\label{eq18}
\begin{cases}\ds
\nabla \phi \neq0 \text{ in } \overline{\Omega\backslash\omega_0}, \\
\ns\ds\frac{\partial\phi}{\partial \nu}\leq0 \text{ on }\partial\Omega, \\
\ns\ds|\nabla\phi(x)\cdot\xi|^2+\sum_{i,j=1}^3(\partial_j
\partial_j\phi(x))\xi_j\xi_j >0\quad
\mbox{in }\overline{\Omega\backslash\omega_0},\;\forall\,\xi=(\xi_1,\xi_2, \xi_3)\in\mathbb{R}^3, \\
\ns\ds\phi(x)>\frac23\|\phi\|_{L^\infty(\Omega)},\qquad
\forall x\in\Omega.
\end{cases}
\end{equation}
Here $\nu=\nu(x)$ denotes the outward normal
vector of $\O$.

There are plenty of choices of $\o_0$ satisfying the above
condition. A typical example can be constructed as
follows.

Let $x_0\in \dbR^3\setminus\overline\Om$ and
$$
\G_0\=\{x\in \G| (x-x_0)\cd \nu(x)\geq 0 \}.
$$
Let $\d>0$. Put
$$
\o_0=O_\d\=\{x\in \Om|\, {\rm dist}(x,\G_0)<\d
\}.
$$
Let $\tilde\psi(\cd)\in C^4(\overline \Om)$ be a
nonnegative function such that
$\tilde\psi(x)=|x-x_0|^2$ for $x\in
\overline{\Omega\backslash\omega_0}$ and 
$\tilde\psi(x)> 0$  for $x\in O_{\d/2}$ and $\tilde{\psi}=0$ on $\Gamma_0$.
Then $\psi(x)=\tilde\psi(x) +
2|\tilde\psi|_{L^\infty(\Om)}$ is the desired
function.

\ms

More examples of such kind of $\o_0$ and $\psi$
can be found in \cite{MOR2008}.

\ms

Clearly, if \eqref{eq18} holds, then there
exists $\o_1\subset\subset\o_0$ such that
\begin{equation}\label{eq18.1}
\begin{cases}\ds
\nabla \phi \neq0 \text{ in }
\overline{\Omega\backslash\omega_1}, \\
\ns\ds|\nabla\phi(x)\cdot\xi|^2+\sum_{i,j=1}^3(\partial_j
\partial_j\phi(x))\xi_j\xi_j >0\quad
\mbox{in
}\overline{\Omega\backslash\omega_1},\;\forall\,\xi=(\xi_1,\xi_2,
\xi_3)\in\mathbb{R}^3.
\end{cases}
\end{equation}
Let $\tilde\omega\in\Omega$ be a neighborhood of
$\omega_1$ such that
$\o_1\subset\subset\tilde\o$ and $\pa\tilde\o$
is $C^2$.  Set
\begin{equation}
\mathcal{H}=C^1(0,T;H^2(\Omega))\cap C^2(0,T;H^1(\Omega)),
\end{equation}
where $H^k(\Omega)$ is the usual Sobolev space.
The Banach space $\mathcal{H}$  is equipped with
its natural norm
\begin{equation}\label{71eq3}
\|z\|^2_\mathcal{H}=\|z\|^2_{C^1(0,T;H^2(\Omega))}+\|z\|^2_{C^2(0,T;H^1(\Omega))},
\q \forall z\in\mathcal{H}.
\end{equation}
Let $\omega\subset\omega_1\subset\Omega$ be an
arbitrary nonempty open subset. Suppose that
$\{a_{jk}\}_{j,k=1}^2\subset L^\infty(Q)$ and
we can choose a constant $a_0>0$ such that 
\begin{equation}\label{522eq5}
a_{21}\geq a_0  \text{ or } -a_{21}\geq a_0
\text{ in } \omega\times(0,T).
\end{equation}
\begin{remark}\label{rem1}
\eqref{522eq5} means that the coupling between
$y_1$ and $y_2$ does not degenerate. More
precisely, $y_1$ can effect $y_2$ adequately.
Without \eqref{522eq5}, one cannot obtain information of $y_2$ from $y_1$.
\end{remark}
Let us now define the admissible set of unknown
coefficients. Fix a constant $M>0$ and two
functions $\varpi_1, \varpi_2\in
L^\infty(\tilde\o;\dbR)$. Let
$\mathcal{A}(\tilde\omega,M)$ be the set of
pairs of real-valued functions $(a_{11},a_{22})$
such that
\begin{equation}\label{71eq1}
\begin{array}{ll}\ds
\mathcal{A}(\tilde\omega,M)\3n&\ds\=\{(a_{11},a_{22})\in
L^\infty(\Omega;\dbR)^2|\,\|a_{jj}\|_{L^\infty(\Omega)}\leq
M, a_{jj}(x)=\varpi_j(x) \text{ on }
\tilde\omega,\\
\ns&\ds\q\;\; \mbox{the equation \eqref{eq1} has
a unique solution $(y_1, y_2) \in\mathcal{H}$
satisfying }\\
\ns&\ds\q\;\; \|y_j\|_\mathcal{H}\leq {\bf C}(M)
\mbox{ for some constant ${\bf C}(M)$  depending
on } M, \; j=1,2\}.
\end{array}
\end{equation}
\begin{remark}
There are mainly two restrictions on a element
in $\mathcal{A}(\tilde\omega,M)$. The first one
is that there is a priori bound $M$. This is
reasonable since in a physical model, one
can assume to know some preliminary upper bound on unknown 
potentials.  The second one is that we know the
value of $(a_{11}(x),a_{22}(x))$ for
$x\in\tilde\o$.  This is technically restrictive
but is acceptable because we may be able to directly measure 
potentials near the boundary.  
Furthermore we note that compared with
\cite{BC2009}, we need less information on unknown potentials.
\end{remark}

In what follows, in order to emphasize the dependence of
the solution to \eqref{eq1} on the unknown
potentials, we write
$(y_1(a_{11},a_{22}),y_2(a_{11},a_{22}))$ for
the solution to \eqref{eq1}.

We choose the initial data $(y_{10},y_{20})$
which satisfy all conditions ensuring that
$\mathcal{A}(\tilde\omega,M)$ is nonempty. Also,
for $j=1,2$, they fulfill 
\begin{equation}\label{78eq1}
\begin{cases}
\ds y_{j0}(x)\in\mathbb{R} \text{ or } iy_{j0}(x)\in\mathbb{R} \text{ a.e. in } \Omega,\\
\ns\ds|y_{j0}(x)|\geq r>0 \text{ a.e. in }\Omega,\\
\ns\ds y_{j}(a_{11},a_{22}) \in
H^1(0,T;L^\infty(\Omega)).
\end{cases}
\end{equation}
\begin{remark}
Condition 
\eqref{78eq1} means that we have to choose initial data suiatably, and 
is a technical restriction.
Similarly to Appendix B in \cite{BC2009}, we can verify that
such $(y_{10},y_{20})$ exists.
\end{remark}

The main result of this paper is stated as
follows.

\begin{theorem}\label{main2D}
There exists a constant $C > 0$ such that
\begin{equation}\label{main2D-eq1}
\begin{array}{ll}\ds
|(a_{11}-\tilde a_{11},a_{22}-\tilde
a_{22})|_{L^2(\O)} \\
\ns\ds\leq C\(\big|\ln \|y_1(a_{11},a_{22})-
\tilde{y}_1(\tilde a_{11},\tilde
a_{22})\|_{H^1(0,T;H^1(\omega))}\big|^{-1}
\\
\ns\ds\qq+\|y_1(a_{11},a_{22})-\tilde y_1(\tilde
a_{11},\tilde
a_{22})\|_{H^1(0,T;H^1(\omega))}\),
\end{array}
\end{equation}
for all $(a_{11},a_{22}), (\tilde a_{11},\tilde
a_{22}) \in \cA(\tilde\o,M)$.
\end{theorem}
\begin{remark}
One can consider the problem that all the coefficients $\{a_{jk}\}_{j,k=1}^2$ are unknown. In this case, the following three conditions are needed: (1) the unknown coefficient $a_{21}$ must be nonzero in a nonempty open subset $\o$; (2) the functions $a_{11}$ and $a_{12}$, $a_{21}$ and $a_{22}$ must be linearity independence, respectively; (3) two times of observations with different suitable chosen initial data of $y_1$ are required. 
As pointed in Remark \ref{rem1}, condition (1) can not be removed since we only observe a single component of the solutions. Condition (2) is reasonable since what we can observe is only the linear combination of the coefficients.  
Condition (3) can not be deleted because for each observation we only observe the linear combinations to get the coefficients from these combinations
and we need observe the system twice.
\end{remark}

\begin{remark}
From the proof of Theorem \ref{main2D}, one can
see that it can be generalized to a system
coupled by more than two Schr\"odinger equations
with an observation on some components of the
solution. In this paper, to present the key idea
in a simple way, we do not pursue the full
technical generality.
\end{remark}

\section{Proof of Theorem \ref{main2D}}

Before giving the proof, we present a
preliminary result.

\begin{lemma}\label{lm1}
For all $(a_{11},a_{22}), (\tilde a_{11},\tilde
a_{22}) \in \cA(\tilde\o,M)$,
\begin{equation}\label{8.7-eq1}
\sum_{j=1}^2\|a_{jj}-\tilde{a}_{jj}\|^2_{L^2(\Omega)}\leq
C\sum_{j=1}^2\|y_j(a_{11},a_{22})-\tilde{y}_j(\tilde{a}_{11},\tilde{a}_{22})\|^2_{H^1(0,T;H^1(\omega_1))}.
\end{equation}
\end{lemma}

In order to obtain the Lipschitz stabilty in \eqref{8.7-eq1}, the subdomain $\omega_1$ can not be arbitrarily small and must satisfy \eqref{eq18.1}. Lemma \ref{lm1}  should be a known result.
However, since we failed to find an exact
reference, we provide it here for the sake of
completeness and readers' convenience.

\begin{proof}[Proof of Lemma \ref{lm1}] Let
$\phi\in C^4(\overline{\Omega})$ be the function
satisfying \eqref{eq18} and \eqref{eq18.1}. Set
\begin{equation}
\hat\f(x,t)\triangleq\frac{e^{\eta\phi(x)}}{(T+t)(T-t)},\q
\hat\a(x,t)\triangleq\frac{e^{2\eta
\|\phi\|_{L^\infty(\Omega)}}-e^{\eta\phi(x)}}{(T+t)(T-t)},\q
 \forall(x,t)\in\Omega\times(0,T),
\end{equation}
where $\eta$ denotes some positive number which
can be specified later.

For $j=1,2$, let
$$
z_j=y_j(a_{11},a_{22})-\tilde{y}_j(\tilde{a}_{11},\tilde{a}_{22}),
\q f_j(x)=a_{jj}(x)-\tilde{a}_{jj}(x), \q
R_j(x,t)=\tilde{y}_j(x,t).
$$
Then $(z_1,z_2) \in
[C([0,T];H^1_0(\Omega))]^2 $ is the solution of
the following system:
\begin{equation}\label{71eq2}
\left\{
\begin{array}{ll}\ds
i\partial_t z_1 +\Delta  z_1 +a_{11}z_1 + a_{12}z_2 =f_1(x)R_1(x,t) & \text{ in }\Omega\times(0,T),\\
\ns\ds i\partial_t z_2 +\Delta  z_2 +a_{21}z_1 + a_{22}z_2 =f_2(x)R_2(x,t) & \text{ in }\Omega\times(0,T),\\
\ns\ds z_1(x,0)=z_2(x,0)=0 &  \text{ in }\Omega,\\
\ns\ds z_1=z_2=0 & \text{ on }\Gamma\times(0,T).
\end{array}
\right.
\end{equation}

Take the even-conjugate extensions of
$(z_1,z_2)$ to the interval
$(-T,T)$, i.e., set
$$
(z_1(x,t),z_2(x,t))=(\overline{z_1(x,-t)},\overline{z_2(x,-t)}) \q\mbox{ for }  t\in(-T,0).
$$
If $(R_1(x,0),R_2(x,0))\in \dbR^2$ for a.e.
$x\in\Omega$, then we set
$$
(R_1(x,t),R_2(x,t))=(\overline{R_1(x,-t)},\overline{R_2(x,-t)})\q\mbox{
for } t\in(-T,0).
$$
If $(iR_1(x,0),iR_2(x,0))\in \dbR^2$ for a.e.
$x\in\Omega$, then we set
$$
(R_1(x,t),R_2(x,t))=(-\overline{R_1(x,-t)},-\overline{R_2(x,-t)})\q\mbox{
for } t\in(-T,0).
$$
In such context, we have that $(R_1,R_2)\in
H^1(-T,T;L^\infty(\Omega))^2$, and
$(z_1,z_2)$ solves the system
\eqref{71eq2} in $\Omega\times(-T,T)$.

Assume $(u_1,u_2) =(\partial_t z_1,
\partial_t z_2)$. We have
\begin{equation}\label{eq2}
\left\{
\begin{array}{ll}\ds
i\partial_t u_1+\Delta u_1+a_{11}u_1 + a_{12}u_2=f_1(x)\pa_tR_1(x,t) & \text{ in }\Omega\times(0,T),\\
\ns\ds i\partial_t u_2+\Delta u_2+a_{21}u_1 + a_{22}u_2=f_2(x)\pa_tR_2(x,t) & \text{ in }\Omega\times(0,T),\\
\ns\ds u_1(x,0)=-if_1(x) R_1(x,0),\q u_2(x,0)=-if_2(x) R_2(x,0) &  \text{ in }\Omega,\\
\ns\ds u_1=u_2=0 &  \text{ on }\Gamma\times(0,T).
\end{array}\right.
\end{equation}

It follows from \eqref{71eq3}, \eqref{71eq2} and
\eqref{eq2} that $(u_1,u_2)\in
[C^1([-T,T];H_0^1(\Omega))\cap C([-T,T];$
$H^2(\Omega))]^2$. Further, there exists a
constant ${\bf C}={\bf C}(M,T)>0$ such that
\begin{equation}\label{52eq1}
\|(u_1,u_2)\|^2_{[L^2(-T,T;H^2(\Omega))]^2} +\|
(\partial_tu_1,\partial_tu_2)\|^2_{[L^2(-T,T;H_0^1(\Omega))]^2}\leq
{\bf C}.
\end{equation}

For $j=1,2$ and $\tau>0$, let $\hat u_j =
e^{-\tau\hat\a} u_j$ and
\begin{equation}\label{8.7-eq3}
\begin{cases}\ds
M_{j1} \= i(2\tau\n\hat\a\cd\n\hat
u_j+\tau\D\hat\a \hat
u_j) + \tau\pa_t\hat\a \hat u_j,\\
\ns\ds M_{j2} \= \pa_t\hat u_j + i(\D\hat u_j +
\tau^2|\n\hat\a|^2\hat u_j).
\end{cases}
\end{equation}
By Proposition 3.1 in \cite{MOR2008}, we know
that there exist $\tau_0>0$ and
$\eta_0(\tau_0)>0$ such that for all $\tau\geq
\tau_0$ and $\eta\geq \eta_0(s_0)$, it holds
that
\begin{equation}\label{8.7-eq2}
\begin{array}{ll}
\ds \int_{-T}^T\int_\O e^{-2\tau\hat\a}
\tau^3\eta^4\hat\f^3
(|u_1|^2+|u_2|^2)dxdt+ \int_{-T}^T\int_\O \sum_{j=1}^2|M_{j2}|^2 dxdt\\
\ns\ds \leq C\Big\{\int_{-T}^T\int_{\o_1}
e^{-2\tau\hat\a} \big[\tau^3\eta^4\hat\f^3
\big(|u_1|^2+|u_2|^2\big)+ \tau\eta^2\hat\f
\big(|\nabla
u_1|^2+|\nabla u_2|^2\big)\big]dxdt\\
\ns\ds \qq\; + \int_{-T}^T\int_\O
e^{-2\tau\hat\a}
\big(|f_1(x)\pa_tR_1(x,t)|^2+|f_2(x)\pa_tR_2(x,t)|^2\big)dxdt\Big\}.
\end{array}
\end{equation}
Put
\begin{equation}\label{8.7-eq2.1}
J = - \int_0^T\int_\O e^{-\tau\hat\a}M_{12} \bar
u_1 dx dt- \int_0^T\int_\O e^{-\tau\hat\a}M_{22}
\bar u_2 dx dt.
\end{equation}
Then
$$
\begin{array}{ll}\ds
\Re(J)&\3n \ds= -\Re\[\int_0^T\int_\Om
 \pa_t\hat u_1 \bar{\hat u}_1 dx dt + i \int_0^T\int_\Om
\big(-|\n\hat u_1|^2 +
\tau^2|\n\hat\a|^2|\hat u_1|^2\big) dx dt\]\\
\ns&\ds \q -\Re\[\int_0^T\int_\Om  \pa_t\hat u_2
\bar{\hat u}_2 dx dt + i \int_0^T\int_\Om
\big(-|\n\hat u_2|^2 +
\tau^2|\n\hat\a|^2|\hat u_2|^2\big) dx dt\]\\
\ns&\ds\3n = \frac{1}{2}\int_\Om \big(|\hat
u_1(x,0)|^2  +|\hat u_2(x,0)|^2\big) dx\\
\ns&\ds\3n = \frac{1}{2}\int_\Om
e^{-2\tau\hat\a(x,0)}\big(|f_1(x)|^2|R_1(x,0)|^2
+|f_2(x)|^2|R_2(x,0)|^2 \big)dxdt.
\end{array}
$$
This, together with the conditions on $R_1(x,0)$
and $R_2(x,0)$, implies that
\begin{equation}\label{8.7-eq5}
\Re(J)\geq \frac{r^2}{2} \int_\Om
e^{-2\tau\hat\a(x,0)}\big(|f_1(x)|^2 +|f_2(x)|^2
\big)dxdt.
\end{equation}
On the other hand, it follows from
\eqref{8.7-eq2.1} that
\begin{equation} \label{8.23-eq1}
\begin{array}{ll}\ds
|J|&\3n\ds \leq \(\int_0^T\int_\O
e^{-2\tau\hat\a} | u_1|^2 dx
dt\)^{\frac{1}{2}}\( \int_0^T\int_\O
|M_{12}|^2 dx dt\)^{\frac{1}{2}}\\
\ns&\ds\q + \(\int_0^T\int_\O e^{-2\tau\hat\a} |
u_2|^2 dx dt\)^{\frac{1}{2}}\( \int_0^T\int_\O
|M_{22}|^2 dx dt\)^{\frac{1}{2}}\\
\ns&\ds\3n\leq
\tau^{\frac{3}{2}}\eta^2\int_0^T\!\int_\O\!
e^{-2\tau\hat\a} \big(|u_1|^2\! +\! |u_2|^2\big)
dx dt \!+\!
\tau^{-\frac{3}{2}}\eta^{-2}\!\int_0^T\!\int_\O\!
\big(|M_{12}|^2\! +\! |M_{22}|^2\big) dx dt.
\end{array}
\end{equation}
From the choice of $\hat\a$, we find that
$$
\begin{array}{ll}\ds
\q\int_{-T}^T\int_\O e^{-2\tau\hat\a}
\big(|f_1(x)\pa_tR_1(x,t)|^2+|f_2(x)\pa_tR_2(x,t)|^2\big)dxdt\\
\ns\ds \leq C\int_\O e^{-2\tau\hat\a(x,0)}
\big(|f_1(x)|^2+|f_2(x)|^2\big)dx.
\end{array}
$$
This, together with \eqref{8.7-eq2},
\eqref{8.7-eq5} and \eqref{8.23-eq1}, implies
that
\begin{equation}\label{8.7-eq6}
\begin{array}{ll}\ds
\q\frac{r^2}{2}\int_\Om
e^{-2\tau\hat\a(x,0)}\big(|f_1(x)|^2 +|f_2(x)|^2
\big)dxdt \\
\ns\ds\leq
C\tau^{-\frac{3}{2}}\eta^{-2}\Big\{\int_{-T}^T\int_{\o_1}
e^{-2\tau\hat\a} \big[
\tau^3\eta^4\hat\f^3\big(|u_1|^2+|u_2|^2\big)+
\tau\eta^2\hat\f \big(|\nabla
u_1|^2+|\nabla u_2|^2\big)\big]dxdt\\
\ns\ds \qq\qq\qq + \int_\O e^{-2\tau\hat\a(x,0)}
\big(|f_1(x)|^2+|f_2(x)|^2\big)dx\Big\}.
\end{array}
\end{equation}
Thus, there is an $\tau_1>0$ such that for all
$\tau\geq \max\{\tau_0,\tau_1\}$ and $\eta\geq
\eta_0(\tau_0)$,
\begin{equation}\label{8.7-eq4}
\begin{array}{ll}
\ds \q\int_\Om
e^{-2\tau\hat\a(x,0)}\big(|f_1(x)|^2 +|f_2(x)|^2
\big)dxdt\\
\ns\ds \leq C\tau^{-\frac{3}{2}}\eta^{-2}
\int_{-T}^T\int_{\o_1} e^{-2\tau\hat\a} \big[
\tau^3\eta^4\hat\f^3 \big(|u_1|^2+|u_2|^2\big)+
\tau\eta^2\hat\f \big(|\nabla u_1|^2+|\nabla
u_2|^2\big) \big]dxdt.
\end{array}
\end{equation}
This concludes \eqref{8.7-eq1} and completes the
proof of Lemma \ref{lm1}.
\end{proof}

Next, in order to keep the self-containment, we
give a brief introduction to F.B.I.
transformation here. Let
$$
F(z)=\frac1{2\pi}\int_\mathbb{R}
e^{iz\varrho}e^{-\varrho^2}d\varrho.
$$
Then
\begin{equation*}
F(z)=\frac{\sqrt{\pi}}{2\pi}e^{\frac14
(|\text{Im} z|^2-|\text{Re} z|^2)}e^{-\frac i2
(\text{Im} z \text{Re} z)}.
\end{equation*}
For every $\lambda\geq1$, define
\begin{equation*}
F_\lambda(z)\triangleq\lambda F(\lambda
z)=\frac{1}{2\pi}\int_\mathbb{R}
e^{iz\varrho}e^{-(\frac{\varrho}{\lambda})^2}
d\tau.
\end{equation*}
Then,
\begin{equation*}
|F_\lambda(z)|=\frac{\sqrt{\pi}}{2\pi}\lambda
e^{\frac{\lambda^2}4 (|\text{Im} z|^2-|\text{Re}
z|^2)}.
\end{equation*}
Let $s, l_0\in\mathbb{R}$, the F.B.I.
transformation  $\mathcal{F}_\lambda$ for
$u\in\mathcal{S}(\mathbb{R}^{n+1})$ is defined
as follows:
\begin{equation}\label{eq19}
(\cF_\l u)(x,s)=\int_{\mathbb{R}}
F_\lambda(l_0+is-l)\Phi(l)u(x,l)dl.
\end{equation}

Now we are in a position to prove Theorem
\ref{main2D}.
\begin{proof}[Proof of Theorem \ref{main2D}] The
proof is long.  We divide it into four steps.

\ms

{\bf Step 1}. In this step, we introduce an
equation on $(-T,T)\times \tilde\o$.

\ms

Recall that $\omega$ is an arbitrary fixed
nonempty subset of $\tilde\omega$ such that
$\overline{\omega}\subset\tilde\omega$. By
\cite[Lemma 1.1]{Fursikov},  there exists a
function $\psi\in C^2(\overline{\tilde\omega})$
such that
\begin{equation}\label{53eq1}
\begin{cases}\ds
\psi(x)>0, \q  \forall x\in\tilde\omega, \\
\ns\ds \psi(x)=0,\q \forall x\in
\partial\tilde\omega, \\
\ns\ds |\nabla\psi(x)|>0, \q  \forall
x\in\overline{\tilde\omega\backslash\omega}.
\end{cases}
\end{equation}
We can conclude from \eqref{53eq1} that there
exist a constant $\beta>0$ and
$\o_2\subset\subset\tilde\o$ such that
\begin{equation}\label{520eq1}
\psi(x)\leq\beta, \q \forall x\in
\tilde\omega\backslash\omega_2
\end{equation}
and that
\begin{equation}\label{520eq3}
\psi(x)\geq2\beta, \q \forall x\in \omega_1.
\end{equation}
It follows from the last condition in
\eqref{53eq1} that the maximum value of $\psi$
can only be attained in $\omega$, i.e., there
exists a point $x_0\in\omega$ such that
\begin{equation}\label{520eq2}
\psi(x_0)=\max_{x\in\tilde\omega}\psi(x).
\end{equation}
Let $\chi\in C^\infty_0(\tilde\o)$ be a cut-off
function, which satisfies $0\leq\chi\leq1$ and
\begin{equation}\label{eq20}
\chi(x)=\left\{
\begin{array}{ll}
1, &\text{ if } x\in \omega_2,\\
0, &\text{ if } x\in \tilde\o\setminus\omega_3,
\end{array}\right.
\end{equation}
where  $\omega_3$ is a subset of $\tilde\o$ such
that $\omega_2\subset\subset\o_3$.

Let $(w_1,w_2)=(\chi u_1,\chi u_2)$. Then by
\eqref{71eq1} and \eqref{eq2}, we have that
\begin{equation}\label{eq21}
\left\{
\begin{array}{ll}\ds
i\partial_t w_1+\Delta w_1 +a_{11}w_1 + a_{12}w_2=[\Delta, \chi]u_1 &\mbox{ in } \tilde\omega\times(0,T),\\
\ns\ds i\partial_t w_2+\Delta w_2 +a_{21}w_1 + a_{22}w_2=[\Delta, \chi]u_2 &\mbox{ in } \tilde\omega\times(0,T),\\
\ns\ds w_1(0)=w_2(0)=0 & \text{ in }\tilde\omega,\\
\ns\ds w_1=w_2 =0 &  \text{ on
}\partial\tilde\omega\times(0,T).
\end{array}\right.
\end{equation}
By \eqref{52eq1}, there exists ${\bf C}={\bf
C}(M,T)>0$ such that
\begin{equation}\label{57eq10}
\|(w_1,w_2)\|^2_{L^2(-T,T;H^2(\widetilde{\omega}))^2}
+\|(\pa_tw_1,\pa_tw_2)\|^2_{L^2(-T,T;H_0^1(\widetilde{\omega}))^2}\leq
{\bf C}.
\end{equation}

\ms

{\bf Step 2}. In this step, we introduce a
system of parabolic equations related to
\eqref{eq21} and a Carleman estimate to the
parabolic system.

\ms

For $j=1,2$, let $\ds W_j(x,s)=\int_{\mathbb{R}}
F_\lambda(l_0+is-l)\Phi(l)w_j(x,l)dl$. Since
$$
\begin{array}{ll}
\partial_s W_j(x,s)\3n&\ds= \int_{\mathbb{R}} -i\partial_l F_\lambda(l_0+is-l)\Phi(l)w_j(x,l)dl \\
\ns&\ds= i\int_{\mathbb{R}}
F_\lambda(l_0+is-l)\left(\Phi'(l)w(x,l)+\Phi(l)w_t(x,l)\right)dl,
\end{array}
$$
we follow that
\begin{equation}\label{eq9}\left\{
\begin{array}{ll}\ds
\partial_s W_1 +\Delta W_1 +a_{11}W_1+a_{12}W_2 =F_1 +G_1  &\mbox{ in }\tilde\omega\times\mathbb{R},\\
\ns\ds \partial_s W_2 +\Delta W_2 +a_{21}W_1+a_{22}W_2 =F_2 +G_2 &\mbox{ in }\tilde\omega\times\mathbb{R},\\
\ns W_1=W_2=0 &\mbox{ on
}\partial\tilde{\omega}\times\mathbb{R},
\end{array}\right.
\end{equation}
where for $j=1,2$,
$$
\begin{array}{ll}\ds
F_j(x,s) =i\int_{\mathbb{R}} F_\lambda(l_0+is-l)\Phi'(l)w_j(x,l)dl,\\
\ns\ds G_j(x,s) =\int_{\mathbb{R}}
F_\lambda(l_0+is-l)\Phi(l)[\Delta, \chi]u_jdl.
\end{array}
$$

Let
\begin{equation}
\varphi(x,t)=\frac{e^{\eta\psi(x)}}{(T+t)(T-t)},\q
\alpha(x,t)=\frac{e^{\eta\psi(x)}-e^{2\eta\|\psi\|_{C(\overline{\tilde\omega})}}}{(T+t)(T-t)},
\q \forall\, (x,t)\in \tilde\o\times(-T,T),
\end{equation}
where $\eta>0$.

Let $\Phi\in C^\infty_0(\mathbb{R})$ satisfying
the following conditions:
$$
\begin{cases}\ds
\Phi\in C^\infty_0\(\[-\frac L2,\frac
L2\];[0,1]\),\\
\ns\ds \Phi=1 \mbox{ on }\[-\frac L4,
\frac{L}4\],\\
\ns\ds |\Phi'|\leq\frac2L,
\end{cases}
$$
where $L>0$ will be chosen later.

Take
$$
K=\[-\frac L2,-\frac
L4\]\bigcup\[\frac{L}4,\frac L2\],\qq
K_0=\[-\frac{L}8,\frac{L}8\].
$$
Then $l_0 \in K_0$ in \eqref{eq19}.

According to Theorem 1.1 in \cite{Gonz2010},
there exist a positive function $\alpha_0\in
C^2(\overline{\tilde{\o}})$ (only depending on
$\tilde{\o}$ and $\omega$), two positive
constants $C_0$ (only depending on $\tilde{\o}$,
$\o$, $\a_0$ and $M_0$) and
$\sigma_0=\sigma_0(\tilde{\o}, \omega, M_0)$
such that the solution $(W_1,W_2)\in
[C([-T,T];L^2(\tilde{\omega}))\cap L^2([-T,T];$
$H^1(\tilde{\omega}))]^2$ of \eqref{eq9}
satisfies that
\begin{equation}\label{78eq2}
\begin{array}{ll}
\ds \int_{-T}^T\int_{\tilde{\omega}}
\big(\sigma^{4}\gamma(s)^{4}|\nabla W_1|^2
+\sigma\gamma(s)|\nabla
W_2|^2+\sigma^{6}\gamma(s)^{6}|W_1|^2
+\sigma^{3}\gamma(s)^{3}|W_2|^2\big)e^{2\sigma\alpha}dxds\\
\ns\ds
\leq C_0 \[\int_{-T}^T\int_{\tilde{\omega}}\big(\sigma^3\gamma(s)^3|F_1(x,s)+G_1(x,s)|^2+|F_2(x,s)+G_2(x,s)|^2\big)e^{2\sigma\alpha}dxds\\
\ns\ds\q\q\q +\sigma^7\int_{-T}^T\int_\omega
e^{2\sigma\alpha}\gamma(s)^7|W_1|^2dxds\],
\end{array}
\end{equation}
where $\gamma(s)= \frac1{(T+s)(T-s)}$ and
$\sigma\geq\sigma_0$.

\ms

{\bf Step 3}. In this step, we estimate all the
terms in the right hand side of \eqref{78eq2}.

\ms

Let
\begin{equation}
\mu_2=\frac{e^{2\eta\psi(x_0)}-e^{\eta\psi(x_0)}}{T^2}.
\end{equation}
There exists  $\delta_2>0$ such that
$$
\max_{x\in\omega,t\in[{-T},T]}
\gamma(s)^7e^{2\sigma\alpha}  \leq
e^{-(2-\delta_2)\sigma\mu_2}.
$$
By the property of F.B.I. transformation, we have
that
\begin{equation}\label{eq23}
\begin{array}{ll}
\ds\q\int_{-T}^T\int_{\omega}\gamma(s)^7|W_1|^2 e^{2\sigma\alpha} dxds\\
\ns\ds\leq \max_{x\in\omega,s\in[{-T},T]} \big(\gamma(s)^7e^{2\sigma\alpha}\big) \int_{-T}^T\int_{\omega}|W_1(x,s)|^2 dxds\\
\ns\ds\leq  e^{-(2-\delta_2)\sigma\mu_2} \int_{-T}^T\int_{\omega}\Big|\int_{\mathbb{R}} F_\lambda(l_0+is-l)\Phi(l)w_1(x,l)dl \Big|^2 dxds\\
\ns\ds\leq  e^{-(2-\delta_2)\sigma\mu_2} \int_{-T}^T\int_{\omega}\Big|\int_{\mathbb{R}} \frac{\sqrt\pi}{2\pi} \lambda e^{\frac{\lambda^2}4(s^2-|l_0-l|^2)}\Phi(l)w_1(x,l)dl \Big|^2  dxds\\
\ns\ds\leq \frac{\lambda^2}{4\pi}e^{-(2-\delta_2)\sigma\mu_2}  \int_{-T}^T e^{\frac{\lambda^2}2 s^2}ds \big|\sup \Phi\big|^2 \int_{\omega}\Big|\int_{-\frac L2}^{\frac L2} w_1(x,l)dl \Big|^2  dx\\
\ns\ds\leq
\frac{\lambda^2LT}{2\pi}e^{-(2-\delta_2))\sigma\mu_2}
e^{\frac{\lambda^2}2 T^2}
\int_{\omega}\int_{-\frac L2}^{\frac L2}
|w_1(x,l)|^2dldx.
\end{array}
\end{equation}
From the definition of $F_j$, we see that
\begin{eqnarray}\label{eq17}
&&\int_{-T}^T\int_{\tilde\omega} |F_j(x,s)|^2 dxds\nonumber\\
&=&\3n\int_{-T}^T\int_{\tilde\omega} \Big|i\int_{\mathbb{R}} F_\lambda(l_0+is-l)\Phi'(l)w_j(x,l)dl\Big|^2 dxds\nonumber\\
&\leq&\int_{-T}^T\int_{\tilde\omega} \Big|\int_{K} \frac{\sqrt{\pi}}{2\pi}\lambda e^{\frac{\lambda^2}4 (s^2-|l_0-l|^2)}\Phi'(l)w_j(x,l)dl\Big|^2 dxds\nonumber\\
&\leq& \frac1{2\pi}\lambda^2 e^{\frac{\lambda^2}2 T^2}T\max_{K} |\Phi'(l)|^2 \int_{\tilde\omega} \Big| \int_{K} e^{-\frac{\lambda^2}4 |l_0-l|^2} w_j(x,l)dl\Big|^2 dx\\
&\leq& \frac1{2\pi}\lambda^2 e^{\frac{\lambda^2}2 T^2}T e^{-\frac{\lambda^2}2 (\frac L8)^2}  \max_{K} |\Phi'(l)|^2  \frac L2\int_{\tilde\omega} \int_{K} |w_j(x,l)|^2 dldx\nonumber\\
&\leq& \frac1{2\pi}\lambda^2 e^{\frac{\lambda^2}2 T^2}T e^{-\frac{\lambda^2}2 (\frac L8)^2}  \left(\frac2L\right)^2  \frac L2\int_{\tilde\omega} \int_{K} |w_j(x,l)|^2 dldx\nonumber\\
&\leq& \frac{\lambda^2T}{\pi L}
e^{\frac{\lambda^2}2 (T^2- (\frac L8)^2)}
\int_{\tilde\omega} \int_{K} |w_j(x,l)|^2
dldx.\nonumber
\end{eqnarray}
Since $\text{supp} \chi'\subset
\tilde\omega\backslash\omega_1$ and $G_j(x,s)=0$
in $\omega_1$, it holds that
\begin{eqnarray}\label{eq16}
&&\int_{-T}^T\int_{\tilde\omega} |G_j(x,s)|^2
dxds\nonumber\\
&=&\int_{-T}^T\int_{\tilde\omega} \Big|\int_{\mathbb{R}} F_\lambda(l_0+is-l)\Phi(l)[\Delta, \chi]u_jdl\Big|^2 dxds\nonumber\\
&\leq&\int_{-T}^T\int_{\tilde\omega} \Big| \int_{-\frac L2}^{\frac L2}\frac{\sqrt{\pi}}{2\pi}\lambda e^{\frac{\lambda^2}4 (s^2-|l_0-l||^2)} [\Delta, \chi]u_jdl\Big|^2 dxds\\
&\leq& \frac1{2\pi}\lambda^2 e^{\frac{\lambda^2}2 T^2}T \int_{\tilde\omega\backslash\omega_1} \Big| \int_{-\frac L2}^{\frac L2}  e^{-\frac{\lambda^2}4 |l_0-l|^2} [\Delta, \chi]u_jdl\Big|^2 dx\nonumber\\
&\leq& \frac{\lambda^2T L}{2\pi}
e^{\frac{\lambda^2}2 T^2}\max\{|\nabla\chi|^2,
|\Delta\chi|^2\} \int_{-\frac L2}^{\frac L2}
\int_{\tilde\omega\backslash\omega_1} \left(
|u_j(x,l)|^2 + |\nabla
u_j(x,l)|^2\right)dxdl.\nonumber
\end{eqnarray}
Set
\begin{equation}
\mu_1=\frac{e^{2\eta\|\psi\|_\infty}-e^{\eta
\beta}}{T^2}.
\end{equation}
By \eqref{520eq1}, we know that there exists
$\delta_1>0$ such that
$$
\max_{x\in\tilde \omega,s\in[{-T},T]}
\gamma(s)^3e^{2\sigma\alpha} \leq
e^{-(2-\delta_1)\sigma\mu_2},\qq
\max_{x\in\tilde \omega\setminus
\o_1,s\in[{-T},T]} \gamma(s)^3e^{2\sigma\alpha}
\leq e^{-(2-\delta_1)\sigma\mu_1}.
$$
Consequently,
\begin{equation}\label{eq25.1}
\begin{array}{ll}
\ds\q\int_{-T}^T\int_{\tilde\omega} e^{2\sigma\alpha}\gamma(t)^3|F_j(x,s)|^2  dxds\\
\ns\ds\leq
e^{-(2-\delta_2)\sigma\mu_2}\frac{\lambda^2T}{\pi
L} e^{\frac{\lambda^2}2 (T^2- (\frac L8)^2)}
\int_{\tilde\omega} \int_{K} |w_j(x,l)|^2 dldx
\end{array}
\end{equation}
and
\begin{equation}\label{eq25}
\begin{array}{ll}
\ds\q\int_{-T}^T\int_{\tilde\omega} e^{2\sigma\alpha}\gamma(t)^3|G_j(x,s)|^2  dxds\\
\ns\ds\leq
e^{-(2-\delta_1)\sigma\mu_1}\frac{\lambda^2T
L}{2\pi}e^{\frac{\lambda^2}2
T^2}\max\{|\nabla\chi|^2,
|\Delta\chi|^2\}\int_{-\frac L2}^{\frac L2}
\int_{\tilde\omega\backslash\omega_1} \left(
|u_j(x,l)|^2 + |\nabla u_j(x,l)|^2\right)dxdl.
\end{array}
\end{equation}

Substituting \eqref{eq23},  \eqref{eq25.1} and
\eqref{eq25}  into \eqref{78eq2}, we obtain that
\begin{eqnarray}\label{eq22.1}
&&\ds  \q\int_{-T}^T\int_{\tilde{\omega}}
\big[\sigma^{4}\gamma(s)^{4}|\nabla W_1|^2
+\sigma\gamma(s)|\nabla
W_2|^2+\sigma^{6}\gamma(s)^{6}|W_1|^2
+\sigma^{3}\gamma(s)^{3}|W_2|^2\big]e^{2\sigma\alpha}dxds\nonumber\\
&&\ds \leq
C_0\[e^{-(2-\delta_1)\sigma\mu_2}\frac{2\lambda^2T}{\pi
L} e^{\frac{\lambda^2}2 (T^2- (\frac L8)^2)}
\int_{\tilde\omega} \int_{K} (\si^3|w_1(x,l)|^2 + |w_2(x,l)|^2) dldx \\
&&\ds\q\q
+e^{-(2-\delta_1)\sigma\mu_1}\!\frac{\lambda^2T
L}{\pi}e^{\frac{\lambda^2}2
T^2}\!\max\{|\nabla\chi|^2, |\Delta\chi|^2\}
\!\!\int_{-\frac L2}^{\frac L2} \!\int_{\tilde\omega\backslash\omega_1}\!\!\!\big[\sigma^3 (|u_1(x,l)|^2 \!\!+\! |\nabla u_1(x,l)|^2)\nonumber\\
&&\ds\q\q\q\q\q\q\q\q+\big( |u_2(x,l)|^2 + |\nabla u_2(x,l)|^2\big)\big]dxdl\nonumber\\
&&\ds\q\q
+\sigma^7\frac{\lambda^2LT}{2\pi}e^{-(2-\delta_2)\sigma\mu_2}
e^{\frac{\lambda^2}2 T^2} \int_{-\frac
L2}^{\frac L2}
\int_{\omega}\left(|w_1(x,l)|^2+ |\nabla
w_1(x,l)|^2\right)dxdl\].\nonumber
\end{eqnarray}

In order to reduce the computation complexity of
the proof, and without loss of generality, in
the following steps, we assume $T=1$. Let $A>1$.
By choosing $L=8AT=8A$, we have
\begin{eqnarray}
&&\ds\3n\3n\q e^{-2\sigma\mu_3}\sigma^4\int_{-1+\epsilon}^{1-\epsilon}\int_{\omega_1} \left(|\nabla W_1|^2+|W_1|^2\right) dxds\nonumber\\
&&\ds\3n\3n\leq  \int_{-1}^1\int_{\tilde{\omega}} [\sigma^{4}\gamma(s)^{4}|\nabla W_1|^2+\sigma^{6}\gamma(s)^{6}|W_1|^2]e^{-2\sigma\alpha}dxds\nonumber\\
&&\ds\3n\3n\leq C_0
\[e^{-(2-\delta_1)\sigma\mu_2}\frac{2\lambda^2}{8\pi
A} e^{\frac{\lambda^2}2 (1- A^2)}
\int_{\tilde\omega} \int_{K} [\sigma^3|w_1(x,l)|^2 +|w_2(x,l)|^2 ]dldx \\
&&\ds\3n\3n\q\q +e^{-(2\!-\delta_1)\sigma\mu_1}\frac{8A\lambda^2}{\pi} e^{\frac{\lambda^2}2}\!\max\{|\nabla\chi|^2, |\Delta\chi|^2\}\!\int_{-4A}^{4A}\! \int_{\tilde\omega\backslash\omega_1}\!\![\sigma ^3 (|u_1(x,l)|^2\!\! + \!|\nabla u_1(x,l)|^2)\nonumber\\
&&\ds\3n\3n\q\q\q\q\q\q\q\q+( |u_2(x,l)|^2 + |\nabla u_2(x,l)|^2)]dxdl\nonumber\\
&&\ds\3n\3n\q\q
+\sigma^7\frac{4A\lambda^2}{\pi}e^{-(2-\delta_2)\sigma\mu_2}
e^{\frac{\lambda^2}2}\int_{-4A}^{4A}
\int_{\omega}\left(|w_1(x,l)|^2+ |\nabla
w_1(x,l)|^2\right)dxdl\],\nonumber
\end{eqnarray}
where
\begin{equation}
\mu_3=\frac{e^{2\eta\|\psi\|_\infty}-e^{2\eta\beta}}{\epsilon(2-\epsilon)}.
\end{equation}
Similarly, we can get that
\begin{eqnarray}
&&\ds\3n\3n\q e^{-2\sigma\mu_3}\sigma \int_{-1+\epsilon}^{1-\epsilon}\int_{\omega_1} \left(|\nabla W_2|^2+|W_2|^2\right) dxds\nonumber\\
&&\ds\3n\3n\leq  \int_{-1}^1\int_{\tilde{\omega}} [\sigma \gamma(s) |\nabla W_2|^2+\sigma^{3}\gamma(s)^{3}|W_2|^2]e^{-2\sigma\alpha}dxds\nonumber\\
&&\ds\3n\3n\leq C_0
\[e^{-(2-\delta_1)\sigma\mu_2}\frac{2\lambda^2}{8\pi
A} e^{\frac{\lambda^2}2 (1- A^2)}
\int_{\tilde\omega} \int_{K} [\sigma^3|w_1(x,l)|^2 +|w_2(x,l)|^2 ] dldx \\
&&\ds\3n\3n\q\q +e^{-(2\!-\delta_1)\sigma\mu_1}\frac{8A\lambda^2}{\pi} e^{\frac{\lambda^2}2}\!\max\{|\nabla\chi|^2, |\Delta\chi|^2\}\!\int_{-4A}^{4A}\! \int_{\tilde\omega\backslash\omega_1}\!\!\big[\sigma^3 \big(|u_1(x,l)|^2\!\! + \!|\nabla u_1(x,l)|^2\big)\nonumber\\
&&\ds\3n\3n\q\q\q\q\q\q\q\q+\big( |u_2(x,l)|^2 + |\nabla u_2(x,l)|^2\big)\big]dxdl\nonumber\\
&&\ds\3n\3n\q\q
+\sigma^7\frac{4A\lambda^2}{\pi}e^{-(2-\delta_2)\sigma\mu_2}
e^{\frac{\lambda^2}2}\int_{-4A}^{4A}
\int_{\omega}\left(|w_1(x,l)|^2+ |\nabla
w_1(x,l)|^2\right)dxdl\].\nonumber
\end{eqnarray}

Fix $\epsilon\in(0,1)$ such that
\begin{equation*}
\tau\triangleq\epsilon(2-\epsilon)\frac{2-\delta_1}2\frac{e^{2\eta\|\psi\|_\infty}-e^{\eta
\beta}}{e^{2\eta\|\psi\|_\infty}-e^{2\eta\beta}}-1>0.
\end{equation*}
This is equivalent to say that
\begin{equation*}
(2-\delta_1)\mu_1-2\mu_3=2\tau\mu_3>0.
\end{equation*}
Hence,
\begin{eqnarray}\label{8.7-eq8}
&&\ds\q
\int_{-1+\epsilon}^{1-\epsilon}\int_{\omega_1}
\left(|\nabla W_1|^2+|W_1|^2\right) dxdt
+ \frac{1}{\si^3}\int_{-1+\epsilon}^{1-\epsilon}\int_{\omega_1} \left(|\nabla W_2|^2+|W_2|^2\right) dxdt\nonumber\\
&&\ds\leq
C_0\frac{1}{\sigma^4}\[\sigma^3e^{-(2-\delta_1)\sigma\mu_2+2\sigma\mu_3}\frac{\lambda^2T}{4\pi
A} e^{\frac{\lambda^2}2 (1- A^2)}
\int_{\tilde\omega} \int_{K}[|w_1(x,l)|^2 +|w_2(x,l)|^2 ] dldx \nonumber\\
&&\ds\q\q +\sigma^3\frac{8A\lambda^2}{\pi} e^{-2\sigma\tau\mu_3}e^{\frac{\lambda^2}2}\max\{|\nabla\chi|^2, |\Delta\chi|^2\}
\big(\|u_1\|^2_{L^2(-4A,4A;H^1(\tilde\omega\backslash\omega_1))}+\nonumber\\
&&\ds \qq +\|u_2\|^2_{L^2(-4A,4A;H^1(\tilde\omega\backslash\omega_1))}\big)  +\sigma^7\frac{4\lambda^2A}{\pi}e^{\sigma(2\mu_3-(2-\delta_2)\mu_2)} e^{\frac{\lambda^2}2} \|w_1\|^2_{L^2(-4A,4A;H^1(\omega))}\]\nonumber\\
&&\ds\leq  C_0\[\frac{\lambda^2}{4\pi\sigma A}e^{-(2-\delta_1)\sigma\mu_2+2\sigma\mu_3}e^{\frac{\lambda^2}2 (1- A^2)} \big( \|w_1\|_{L^2(K\times\tilde\omega)}^2+ \|w_2\|_{L^2(K\times\tilde\omega)}^2\big) \nonumber\\
&&\ds\q\q +\frac{8\lambda^2 A}{\pi\sigma}e^{-2\sigma\tau\mu_3}e^{\frac{\lambda^2}2}\max\{|\nabla\chi|^2, |\Delta\chi|^2\}(\|u_1\|^2_{L^2(-4A,4A;H^1(\tilde\omega\backslash\omega_1))}\\
&&\ds\q\q
+\|u_2\|^2_{L^2(-4A,4A;H^1(\tilde\omega\backslash\omega_1))})
+\frac{4\lambda^2\sigma^3A}{\pi}e^{\sigma(2\mu_3-(2-\delta_2)\mu_2)}  e^{\frac{\lambda^2}2} \|w_1\|^2_{L^2(-4A,4A;H^1(\omega))}\]\nonumber\\
&&\ds\leq  C_0\frac{8\lambda^2 A}{\pi\sigma}\left[\frac1{32A^2}e^{-(2-\delta_1)\sigma\mu_2+2\sigma\mu_3}e^{\frac{\lambda^2}2 (1- A^2)} \big( \|w_1\|_{L^2(K\times\tilde\omega)}^2+ \|w_2\|_{L^2(K\times\tilde\omega)}^2\big)\right.\nonumber\\
&&\ds\q\q +e^{-2\sigma\tau\mu_3}e^{\frac{\lambda^2}2}\max\{|\nabla\chi|^2, |\Delta\chi|^2\}\big(\|u_1\|^2_{L^2(-4A,4A;H^1(\tilde\omega\backslash\omega_1))}\nonumber\\
&&\ds\qq +\|u_2\|^2_{L^2(-4A,4A;H^1(\tilde\omega\backslash\omega_1))}\big)\nonumber \\
&&\ds\q\q\left.+\frac{\sigma^4}2e^{\sigma(2\mu_3-(2-\delta_2)\mu_2)}  e^{\frac{\lambda^2}2 T^2} \|w_1\|^2_{L^2(-4A,4A;H^1(\omega))}\right]\nonumber\\
&&\ds\leq  C_0\!\frac{16\lambda^2 \!A}{\pi\sigma}\left[\left(e^{-(2-\delta_1)\sigma\mu_2+2\sigma\mu_3}\frac1{32A^2}e^{\frac{\lambda^2}2 (1- A^2)}\!+\!e^{-2\sigma\tau\mu_3}e^{\frac{\lambda^2}2}\!\max\{|\nabla\chi|^2, |\Delta\chi|^2\}\right){\bf C}\right.\nonumber\\
&&\ds\q\q\left.+\frac{\sigma^4}{2
}e^{\sigma(2\mu_3-(2-\delta_2)\mu_2)}
e^{\frac{\lambda^2}2}
\|w_1\|^2_{L^2(-4A,4A;H^1(\o))}\right].\nonumber
\end{eqnarray}

\noindent{\bfseries\itshape Step 4.} By Lemma
\ref{lm1},
\begin{equation}\label{eq27}
\begin{array}{ll}
\ds\q\sum_{j=1}^2\|a_{jj}-\tilde{a}_{jj}\|_{L^2(\Omega)}\\
\ns\ds\leq C\|(y_1(a_{11},a_{22})-\tilde{y}_1(\tilde{a}_{11},\tilde{a}_{22}),y_2(a_{11},a_{22})-\tilde{y}_2(\tilde{a}_{11},\tilde{a}_{22}))\|_{[H^1(0,T;H^1(\omega))]^2}\\
\ns\ds\leq  C\left(\|(\nabla w_1,\nabla
w_2)\|_{L^2(\omega_1\times(-T,T))}^2
+\|(w_1,w_2)\|_{L^2(\omega_1\times(-T,T))}^2\right)\\
\ns\ds\leq C\left(\|\big(\nabla (\Phi
w_1),\nabla (\Phi
w_2)\big)\|_{L^2(\omega_1\times(-\frac L2,\frac
L2))}^2+\|\big(\Phi w_1,\Phi
w_2\big)\|_{L^2(\omega_1\times(-\frac L2,\frac
L2))}^2\right).
\end{array}
\end{equation}
It follows from Parseval's identity that
\begin{equation}\label{eq27}
\begin{array}{ll}
\ds\q\|\Phi w_j\|_{L^2(\omega_1\times(-\frac L2,\frac L2))}^2=\int_{-\frac L2}^{\frac L2}\int_{\omega_1} |\Phi(t)w_j(x,t)|^2 dxdt\\
\ns\ds=\int_{\mathbb{R}}\int_{\omega_1} |\Phi(t)w_j(x,t)|^2 dxdt=\frac{1}{2\pi}\int_{\mathbb{R}}\int_{\omega_1} |\widehat{\Phi(l_0)w_j}(x,l_0)(t)|^2 dxdt\\
\ns\ds\leq\frac{1}{2\pi}\int_{\mathbb{R}}\int_{\omega_1}
|(1-F_\lambda)\widehat{\Phi(l_0)w_j}(x,l_0)(t)|^2
dxdt+\int_{\mathbb{R}}\int_{\omega_1}
|F_\lambda*\Phi(\cdot)w_j(x,\cdot)(l_0)|^2
dxdl_0.
\end{array}
\end{equation}
The first term in the right hand side of
\eqref{eq27} reads
\begin{eqnarray}\label{eq28}
&&\ds\frac{1}{2\pi}\int_{\mathbb{R}}\int_{\omega_1} |(1-F_\lambda)\widehat{\Phi(l_0)w_j}(x,l_0)(t)|^2 dxdt\nonumber\\
&&=\frac{1}{2\pi}\int_{\mathbb{R}}\int_{\omega_1} (1-e^{-(\frac t\lambda)^2})^2|\widehat{\Phi(l_0)w_j}(x,l_0)(t)|^2 dxdt\nonumber\\
&&\ds\leq\frac{1}{\pi\lambda^2}\int_{\mathbb{R}}\int_{\omega_1} |t\widehat{\Phi(l_0)w_j}(x,l_0)(t)|^2 dxdt\nonumber\\
&&\ds\leq\frac{2}{\lambda^2}\int_{\mathbb{R}}\int_{\omega_1} |\Phi'(l_0)w_j(x,l_0)+\Phi(l_0)\partial_{l_0} w_j(x,l_0)|^2 dxdl_0\\
&&\ds\leq\frac{4}{\lambda^2}\int_{\mathbb{R}}\int_{\omega_1}\left( |\Phi'(l_0)w_j(x,l_0)|^2+\Phi(l_0)\partial_{l_0} w_j(x,l_0)|^2\right) dxdl_0\nonumber\\
&&\ds\leq\frac{4}{\lambda^2}\[\(\frac2L\)^2\int_{K_0}\int_{\omega_1}|w_j(x,l_0)|^2dxdl_0+\int_0^L\int_{\omega_1} |\partial_{l_0}w_j(x,l_0)|^2 dxdl_0\]\nonumber\\
&&\ds\leq\frac{4}{\lambda^2}\[\(\frac1{4AT}\)^2\int_{K_0}\int_{\omega_1}|w_j(x,l_0)|^2dxdl_0+\int_0^{8AT}\int_{\omega_1}
|\partial_{l_0}w_j(x,l_0)|^2 dxdl_0\].\nonumber
\end{eqnarray}

Let
\begin{equation}\label{512eq1}
W_{j,\lambda}(x,l_0)\triangleq
W_j(x,0)=\int_{\mathbb{R}}
F_\lambda(l_0-l)\Phi(l)w_j(x,l)dl=F_\lambda*\Phi(\cdot)w_j(x,\cdot)(l_0).
\end{equation}
By applying the Cauchy integral formula, for
$\rho\in(0,T-\epsilon)$ and by setting
$z=\kappa+\rho e^{i\phi}$, we have that
\begin{equation}\label{512eq2}
\begin{array}{ll}
\ds W_{j,\lambda}(x,\kappa)=\frac1{2\pi i}\int_{|z-\kappa|=\rho}\frac{W_{j,\lambda}(x,z)}{z-\kappa}dz=\frac1{2\pi i}\int_0^{2\pi}W_{j,\lambda}(x,\kappa+\rho e^{i\phi})d\phi\\
\ns\ds=\frac1{2\pi i (T-\epsilon)}\int_0^{T-\epsilon}\int_0^{2\pi}W_{j,\lambda}(x,\kappa+\rho e^{i\phi})d\phi d\rho\\
\ns\ds=\frac1{2\pi i (T-\epsilon)}\int_{-T+\epsilon}^{T-\epsilon}\int_{-\sqrt{(T-\epsilon)^2-l_0^2}}^{\sqrt{(T-\epsilon)^2-l_0^2}}W_{j,\lambda}(x,l_0+is)|J(l_0,s)|dsdl_0\\
\ns\ds=\frac1{2\pi i
(T-\epsilon)}\int_{-T+\epsilon}^{T-\epsilon}\int_{-\sqrt{(T-\epsilon)^2-l_0^2}}^{\sqrt{(T-\epsilon)^2-l_0^2}}W_j(x,s)dsdl_0.
\end{array}
\end{equation}
Thus,
\begin{equation}\label{512eq2.2}
\begin{array}{ll}
\ds |W_{j,\lambda}(x,\kappa)|^2
=\frac{1}{\pi^2}\int_{-T+\epsilon}^{T-\epsilon}\int_{-T+\epsilon}^{T-\epsilon}\big|W_j(x,s)\big|^2dsdl_0.
\end{array}
\end{equation}
Integrating \eqref{512eq2.2} with respect to $x$
over $\omega_1$ and with respect to $\kappa$
over $[-\frac L2,\frac L2]$, we get that
\begin{equation}\label{512eq2}
\begin{array}{ll}
\ds\int_{-4AT}^{4AT}\int_{\omega_1} |W_{j,\lambda}(x,\kappa)|^2dxd\kappa\\
\ns\ds\leq\frac{1}{\pi (1-\epsilon)^2} \int_{-4AT}^{4AT}\int_{-T+\epsilon}^{T-\epsilon}\left(\int_{-T+\epsilon}^{T-\epsilon}\int_{\omega_1}\big|W_j(x,s)\big|^2dxds\right)dl_0d\kappa\\
\ns\ds\leq\frac{16AT(T-\epsilon)}{\pi^2}
\int_{-T+\epsilon}^{T-\epsilon}\int_{\omega_1}\big|W_j(x,s)\big|^2dxds.
\end{array}
\end{equation}
Substituting \eqref{eq28}, \eqref{512eq2} into \eqref{eq27} and noting that $T=1$, we find that
\begin{equation}\label{514eq1}
\begin{array}{ll}
\ds\q\|\Phi w_j\|_{L^2(\omega_1\times(-4A,4A))}^2\\
\ns\ds\leq\frac{4}{\lambda^2}\[\frac1{16A^2}\int_{K_0}\int_{\omega_1}|w_j(x,l_0)|^2dxdl_0+\int_{-4A}^{4A}\int_{\omega_1} |\partial_{l_0}w_j(x,l_0)|^2 dxdl_0\]\\
\ns\ds\q\q+\frac{16A}{\pi^2}
\int_{-1+\epsilon}^{1-\epsilon}\int_{\omega_1}\big|W_j(x,s)\big|^2dxds.
\end{array}
\end{equation}
Similarly, we can obtain that
\begin{equation}\label{514eq3}
\begin{array}{ll}
\ds\q\|\nabla(\Phi w_j)\|_{L^2(\omega_1\times(-4A,4A))}^2\\
\ns\ds\leq\frac{4}{\lambda^2}\[\frac1{16A^2}\int_{K_0}\int_{\omega_1}|\nabla w_j(x,l_0)|^2dxdl_0+\int_{-4A}^{4A}\int_{\omega_1} |\partial_{l_0}\nabla w_j(x,l_0)|^2 dxdl_0\]\\
\ns\ds\q\q+\frac{16A}{\pi^2}
\int_{-1+\epsilon}^{1-\epsilon}\int_{\omega_1}\big|\nabla
W_j(x,s)\big|^2dxds.
\end{array}
\end{equation}
Let $\ds\sigma=\frac{\lambda^2}{2\tau\mu_3}$ and
$C_2=\max\{|\nabla\chi|^2, |\Delta\chi|^2\}$ such that
$$
e^{\sigma(2\mu_3-(2-\delta_2)\mu_2)}e^{\frac{1-A^2}2}\leq
1.
$$
From \eqref{8.7-eq8}, \eqref{eq27},
\eqref{514eq1} and \eqref{514eq3}, we have
\begin{eqnarray}\label{514eq2}
&&\ds\q\sum_{j=1}^2\|a_{jj}-\tilde{a}_{jj}\|^2_{L^2(\Omega)} \nonumber\\
&&\ds\leq C\bigg\{\frac{4}{\lambda^2}\[ \frac1{16A^2} \int_{K_0}\int_{\omega_1}\sum_{j=1}^2\big(|w_j(x,l_0)|^2+|\nabla w_j(x,l_0)|^2\big)dxdl_0  \nonumber\\
&&\ds\q\q  +\int_{-4A}^{4A}\int_{\omega_1}\sum_{j=1}^2 (|\partial_{l_0}w_j(x,l_0)|^2 + |\partial_{l_0}\nabla w_j(x,l_0)|^2 )dxdl_0\]\nonumber\\
&&\ds\q\q +\frac{1}{\pi^2}\(1 + \frac{1}{1 - \epsilon}\)^2 C_0 \frac{256A^2\lambda^2}{\pi \sigma}\[ \(\frac1{32A^2}e^{\sigma(2\mu_3-(2-\delta_2)\mu_2)}e^{\frac{1-A^2}2 \lambda^2}\\
&&\ds \qq + e^{-2\sigma\tau\mu_3}e^{\frac{\lambda^2}2}\max\{|\nabla\chi|^2, |\Delta\chi|^2\}\){\bf C}\nonumber\\
&&\ds\q\q
+\frac{\sigma^4}{2}e^{\sigma(2\mu_3-(2-\delta_2)\mu_2)}
e^{\frac{\lambda^2}2}
\|w_1\|^2_{L^2(-4A,4A;H^1(\omega))}\]
\bigg\}\nonumber\\
&&\ds\leq C\bigg\{\[\frac{8}{\lambda^2}+ C_0\frac{256A^2\lambda^2}{\pi^3 \sigma}e^{-2\sigma\tau\mu_3}\(\frac1{32A^2}e^{-\frac{A^2-1}2 \lambda^2} +C_2e^{\frac{\lambda^2}2}\)\]{\bf C} \nonumber\\
&&\ds\q\q +C_0 \frac{128A^2\lambda^2\sigma^3}{\pi^3 }e^{\sigma(2\mu_3-(2-\delta_2)\mu_2)}  e^{\frac{\lambda^2}2} \|w_1\|^2_{L^2(-4A,4A;H^1(\omega))}\bigg\}\nonumber\\
&&\ds\leq C\bigg\{\[\frac{8}{\lambda^2}+ C_0\frac{512A^2\tau\mu_3}{\pi^3}\(\frac1{32A^2}e^{-\frac{A^2+1}2 \lambda^2} +C_2e^{-\frac{\lambda^2}2}\)\]   {\bf C} \nonumber\\
&&\ds\q\q
+C_0
\frac{16\tau^3\mu_3^3A^2}{\pi^3
\lambda^4}e^{(\frac{2\mu_3-(2-\delta_2)\mu_2}{2\tau\mu_3}+\frac{1}2)\lambda^2}
\|w_1\|^2_{L^2(-4A,4A;H^1(\omega))}\bigg\}.\nonumber
\end{eqnarray}
Let $\lambda\geq\lambda_0$ be such that
\begin{equation}\label{519eq1}
\sum_{j=1}^2\|a_{j}-\tilde{a}_{j}\|^2_{L^2(\Omega)}
\leq \frac{C_3}{\lambda^2}{\bf
C}+e^{C_4{\lambda^2}}
\|w_1\|^2_{L^2(-T,T;H^1(\omega))},
\end{equation}
where $C_3$ and $C_4$ are two constants
independent of $\lambda$. Taking
$$
\lambda=\max\Big\{\lambda_0,\(\frac { |\ln
\|w_1\|_{L^2(-T,T;H^1(\omega))}|}
{C_4}\)^\frac12\Big\}.
$$
If $\|w_1\|_{L^2(0,T;H^1(\omega))}$ is small
enough, then
\begin{equation}\label{519eq2}
\begin{array}{ll}
\ds\sum_{j=1}^2\|a_{jj}-\tilde{a}_{jj}\|^2_{L^2(\Omega)}
\3n&\ds\leq \frac{C_3C_4}{ |\ln \|w_1\|_{L^2(-T,T;H^1(\omega))}|}{\bf C}+\|w_1\|_{L^2(-T,T;H^1(\omega))}\\
\ns&\ds\leq C\left( |\ln
\|w_1\|_{L^2(-T,T;H^1(\omega))}|^{-1}{\bf
C}+\|w_1\|_{L^2(-T,T;H^1(\omega))}\right)\\
\ns&\ds\leq C\left( |\ln
\|w_1\|_{L^2(0,T;H^1(\omega))}|^{-1}{\bf
C}+\|w_1\|_{L^2(0,T;H^1(\omega))}\right).
\end{array}
\end{equation}
Otherwise, there exists a constant $m>0$ such
that $\|w_1\|_{L^2(-T,T;H^1(\omega))}\geq m$.
Thus, by \eqref{57eq10} we have
\begin{equation}\label{519eq3}
\sum_{j=1}^2\|a_{jj}-\tilde{a}_{jj}\|^2_{L^2(\Omega)}
\leq {\bf C}=\frac {\bf C}mm\leq
C\|w_1\|_{L^2(-T,T;H^1(\omega))}\leq
C\|w_1\|_{L^2(0,T;H^1(\omega))}.
\end{equation}
\end{proof}
\section*{acknowledgement}
The first author thanks the support of the National Natural Science Foundation of China (No. 11501086), the Fundamental Research Funds for the Central Universities (No. ZYGX2016J137) and the Science Strength Promotion Programme of UESTC. The second author is supported by Grant-in-Aid for Scientific Research (S) 15H05740 and A3 Foresight Program Modeling and Computation of Applied Inverse Problems” of Japan Society for the Promotion of Science, and the ”RUDN University Program 5-100”.

\end{document}